\documentclass{amsart}
\numberwithin{equation}{section}

\usepackage{bbm}
\usepackage{comment}
\usepackage{verbatim}
\usepackage{mathrsfs}
\usepackage{mathtools}
\usepackage{latexsym}
\usepackage{epsfig}
\usepackage{amsmath}
\usepackage[usenames,dvipsnames]{xcolor}
\usepackage{graphics}
\usepackage[utf8]{inputenc}
\usepackage{amssymb}
\usepackage{amsthm}
\usepackage[alphabetic]{amsrefs}
\usepackage{amsopn}
\usepackage{amscd}
\usepackage[all,knot]{xy}
\xyoption{all}
\usepackage{rotating}
\usepackage{tikz}
\usetikzlibrary{calc}
\usepgflibrary{shapes.geometric}
\usepgflibrary{shapes.misc}
\usetikzlibrary{positioning}
\usetikzlibrary{decorations}
\usetikzlibrary{decorations.pathreplacing}
\usepackage{hyperref}
\usepackage{tikz}
\usepackage{tikz-cd}
\theoremstyle{plain}
\newtheorem{thm}{Theorem}[subsection]
\newtheorem{lem}[thm]{Lemma}
\newtheorem{prop}[thm]{Proposition}
\newtheorem{cor}[thm]{Corollary}

\newtheorem*{thm*}{Theorem}
\newtheorem*{lem*}{Lemma}
\newtheorem*{prop*}{Proposition}
\newtheorem*{cor*}{Corollary}

\theoremstyle{definition}
\newtheorem{defn}[thm]{Definition}
\newtheorem*{defn*}{Definition}

\newtheorem{ex}[thm]{Example}
{}
\newtheorem{rem}[thm]{Remark}
\newtheorem*{rem*}{Remark}

{}
{}
{}
\newtheorem*{ack}{Acknowledgements}{}
{}

\theoremstyle{remark}
{}
{}
{}

\def\to{\longrightarrow} %I hate short arrows

\def\CC{\mathbb{C}}

\def\PP{\mathbb{P}}

\def\RR{\mathbb{R}}

\def\ZZ{\mathbb{Z}}
\DeclareUnicodeCharacter{221E}{$\infty$}

\def\sfA{\mathsf{A}}
\def\sfB{\mathsf{B}}
\def\sfC{\mathsf{C}}
\def\sfD{\mathsf{D}}

\def\sfM{\mathsf{M}}

\def\mcF{\mathcal{F}}

\def\mcK{\mathcal{K}}
\def\mcL{\mathcal{L}}

\def\mcO{\mathcal{O}}

\def\mcT{\mathcal{T}}

\def\mfm{\mathfrak{m}}

\def\mfp{\mathfrak{p}}

\def\Id{\mathrm{Id}}

\DeclareMathOperator{\add}{add}

\DeclareMathOperator{\Hom}{Hom}

\DeclareMathOperator{\modu}{\mathsf{mod}}
\DeclareMathOperator{\Modu}{\mathsf{Mod}}
\DeclareMathOperator{\comodu}{\mathsf{comod}}
\DeclareMathOperator{\Comodu}{\mathsf{Comod}}
\DeclareMathOperator{\smodu}{\underline{\mathsf{mod}}}
\DeclareMathOperator{\sModu}{\underline{\mathsf{Mod}}}

\DeclareMathOperator{\Ext}{Ext}

\DeclareMathOperator{\Tor}{Tor}

\DeclareMathOperator{\Spc}{Spc}

\def\SH{\mathsf{SH}}
\DeclareMathOperator{\Ho}{Ho}
\def\cyc{A}

\definecolor{internationalkleinblue}{rgb}{0.0, 0.18, 0.65}

\title{Homological residue fields as comodules over coalgebras}

\author{James C.\ Cameron}
\address{James C.\ Cameron, Department of Mathematics,
University of Utah,
155 S 1400 E
Salt Lake City, UT 84112
}
\email{cameron@math.utah.edu}
\urladdr{http://www.math.utah.edu/~cameron}

\author{Greg Stevenson}
\address{Greg Stevenson, Aarhus University, Department of Mathematics, Ny Munkegade 118, bldg. 1530
DK-8000 Aarhus C, Denmark
}
\email{greg@math.au.dk}
\urladdr{https://sites.google.com/view/gregstevenson}


\begin{document}

\begin{abstract}
\noindent We explicitly present homological residue fields for tensor triangulated categories as categories of comodules in a number of examples across algebra, geometry, and topology. Our results indicate that, despite their abstract nature, they are very natural objects and encode tangent data at the corresponding point on the spectrum.
\end{abstract}

\maketitle

%\setcounter{tocdepth}{1}
%\tableofcontents

%--------------------------------------------------------------------------------------------------------------------------------------------------

%--------------------------------------------------------------------------------------------------------------------------------------------------

\section{Introduction}
Objects that play the role of fields are crucial to many computations in tensor triangular geometry (tt-geometry for short). They have been used in computing the Balmer spectrum in several examples, in extending support theories to ind-completions and classifying localizing subcategories, and in formulating and proving nilpotence statements. But, as always, there is a catch: we do not know that the tensor triangular analogue of residue fields exist in general, and we don't know how to abstractly construct the ones we have in examples.

Homological residue fields, introduced in \cite{BKS}, are an abelian substitute for tensor triangular residue fields which are engineered to describe a shadow of the coveted triangulated analogue. They exist in glorious generality, and have proved to be sufficient stand-ins for many purposes. For example, Balmer has used them to define a notion of support for arbitrary objects of a big tensor triangulated category \cite{balmerhomologicalsupports} and to detect the nilpotence of maps locally \cite{Balmer20b}.

These homological residue fields have many desirable properties, amongst them the advantage that their construction is uniform for all tensor triangulated categories. But, until recently, their precise form has been a mystery even in the simplest non-trivial examples. 

A first step toward describing these categories explicitly was given in \cite{BC20} where, in the presence of a tt-residue field, the homological residue field is presented as the category of comodules over an associated comonad. The main aim of this paper is to show that, in favourable situations, this comonad is simply tensoring with a bialgebra, and thus the homological residue field is, as a monoidal abelian category, the category of comodules for the corresponding bialgebra with the induced tensor product operation. This is encapsulated in the following theorem which summarises Theorem~\ref{thm:thepoint} and Corollary~\ref{cor:monoidal} applied to the tt-setting.

\begin{thm*}
Suppose that a big tensor triangulated category $\mcT$ admits a residue field functor $F\colon \mcT \to \mcK$, and let $C$ denote the associated comonad on $\Modu \mcK^c$. Assume that $\Modu \mcK^c$ is equivalent as a graded monoidal category to $\Modu^\ZZ R$ for a graded ring $R$. Then $C(R)$ is a bialgebra, $C$ is equivalent to $-\otimes_R C(R)$, and the homological residue field associated to $F$ is monoidally equivalent to $R$-linear comodules for the bialgebra $C(R)$.
\end{thm*}

We can compute this bialgebra completely explicitly in many examples and we find objects that are already of interest and well studied.

%; for the derived category of a commutative ring $R$ and residue field $k$ the bialgebra is $\Tor^R_*(k,k)$, and in the stable homotopy category the homological residue field associated to a Morava $K$-theory spectrum $K(p,n)$ is comodules for $K(p,n)_*K(p,n)$, the Hopf algebra of cooperations for $K(p,n)$-homology. 

\begin{enumerate}
\item For a commutative noetherian ring $R$ and prime ideal $\mfp$ the homological residue field associated to the residue field $k(\mfp)$ is the category of graded comodules for the graded Hopf algebra $\Tor^R_*(k(\mfp),k(\mfp))$ (Section~\ref{ssec:alg}).
\item In the stable homotopy category: the homological residue field associated to the Morava $K$-theory $K(p,n)$ is the category of graded comodules over the graded Hopf algebra of cooperations on $K(p,n)$-homology (Section~\ref{ssec:top}).
\item In modular representation theory: if $k$ is a field of characteristic $2$ and $E$ is an elementary abelian 2-group of rank $r$, then the homological residue field corresponding to a cyclic shifted subgroup $k[x]/(x^2) \to kE$ is the category of comodules over the (ungraded) exterior algebra on $r-1$ generators with the usual Hopf algebra structure (Section~\ref{ssec:rep}).
\end{enumerate}

Different tt-categories can of course have the same tt-residue field (as categories), in much the same way as every closed point of any affine space over $\CC$ has residue field $\CC$. On the other hand, a prime ideal of a tt-category comes, by definition, with an embedding. The homological residue field remembers something about its origins, and the examples suggest this ancestral memory is tangent information at the chosen point.

The final sections of the paper discuss invariance of the homological residue field under changing the tt-field that we map to, i.e.\ a form of descent, and some standard duality statements that let us translate from comodules to modules in sufficiently regular settings.

\begin{ack}
We are grateful to Paul Balmer for precious discussions and his encouragement to pursue this idea, and to Scott Balchin and Paul Balmer for valuable comments on an earlier version of this manuscript.
\end{ack}

%--------------------------------------------------------------------------------------------------------------------------------------------------

%--------------------------------------------------------------------------------------------------------------------------------------------------

\section{Preliminaries}
\setcounter{subsection}{1}

We shall begin by reminding ourselves of the construction of the comonad associated to a tt-field and the connection to the corresponding homological residue field. This is a restriction: we don't know that tt-residue fields exist in general. However, the comonad on the tt-field will be key to our description and so we remain preoccupied with this setting. For full details on homological residue fields, and the other constructions used here, we refer to \cite{BKS}.

Let $\mcT$ be a big tt-category, i.e.\ $\mcT$ is a rigidly-compactly generated tensor triangulated category, and $\mcK$ a tt-field, i.e.\ $\mcK$ is a big tt-category which is not zero and such that for every $X\in \mcK$ the internal hom-functor $\hom(-,X)$ is faithful (again see \cite{BKS} for details, in particular Definition~1.1 and Theorem~5.21).

Given an exact coproduct preserving monoidal functor $F\colon \mcT \to \mcK$ there is an induced commutative square 
\[
\xymatrix{
\mcT \ar[r] \ar[d]_-{F} & \Modu \mcT^c  \ar[d]^-{\hat{F}}  \\
\mcK \ar[r] & \Modu \mcK^c
}
\]
where $\Modu \mcT^c$ (resp.\ $\Modu \mcK^c$) are the categories of contravariant additive functors from $\mcT^c$ (resp.\ $\mcK^c$) to abelian groups, the horizonal arrows are the restricted Yoneda functors, and $\hat{F}$ is the left Kan extension of $F$ composed with the restricted Yoneda functor. The categories $\Modu \mcT^c$ and $\Modu \mcK^c$ are symmetric monoidal Grothendieck categories, all functors occurring are monoidal and colimit preserving, and $\hat{F}$ is exact. The \emph{homological residue field associated to} $F\colon \mcT \to \mcK$ is $\Modu \mcT^c/ \ker \hat{F}$.

As $F$ and $\hat{F}$ are colimit preserving they admit right adjoints $U$ and $\hat{U}$ respectively (and as the notation suggests $\hat{U}$ is induced by $U$). This induces an adjunction relating the homological residue field and $\Modu \mcK^c$ and factoring the adjunction $\hat{F} \dashv \hat{U}$ through the coreflective localization $Q$ giving the homological residue field.
\[ \xymatrix{
\Modu \mcK^c \ar@<-.5ex>[r]_-{\bar{U}}  	& \Modu \mcT^c/ \ker \hat{F} \ar@<-.5ex>[l]_-{\bar{F}} \ar@<-.5ex>[r] &  \ar@<-.5ex>[l]_-{Q} \Modu \mcT^c
}.
\]
Our principal interest will be in the comonad $\hat{F}\hat{U} \cong \bar{F}\bar{U}$ which gives a description of the homological residue field.
\begin{thm}[\cite{BC20}]
The homological residue field $\Modu \mcT^c/ \ker \hat{F}$ is equivalent to comodules for the comonad $\bar{F}\bar{U}$.
\end{thm}

\begin{rem}
The category $\Modu \mcT^c/ \ker \hat{F}$ is symmetric monoidal, but it is not clear from the theorem how this structure is taken into account. In Section~\ref{sec:monoidal} we make the observation that $\bar{F}\bar{U}$ is a monoidal comonad and this induces a symmetric monoidal structure on its category of comodules making the equivalence of the theorem monoidal.
\end{rem}
%--------------------------------------------------------------------------------------------------------------------------------------------------

%--------------------------------------------------------------------------------------------------------------------------------------------------

\section{From comonads to corings}
\setcounter{subsection}{1}

Let $R$ be a $\ZZ$-graded ring. Suppose that we are given a comonad $C$ on $\Modu^\ZZ R$ the category of graded (right) $R$-modules. Let $(1)$ denote the grading shift autoequivalence on $\Modu^\ZZ R$. We assume throughout that $C$ commutes with $(1)$ and preserves colimits.

\begin{rem}\label{rem:conditions}
These conditions are quite natural; they are automatically satisfied for the comonad associated to the adjunction between base change and restriction of scalars along a morphism of graded rings. In particular, it covers $\hat{F}\hat{U}$ of the previous section in several examples.
\end{rem}

The category $\Modu^\ZZ R$ carries a canonical enrichment over the category of graded abelian groups, via
\[
\hom_R(M,N) = \bigoplus_{i\in \ZZ} \Hom_R(M, N(i)).
\]
As the enrichment is canonical, and is determined by the autoequivalence $(1)$ on the underlying category, we will not introduce special notation to distinguish between the enriched and usual incarnations of $\Modu^\ZZ R$.

We have assumed $C$ is additive and commutes with the grading shift. It follows that $C$ is compatible with the enrichment over graded abelian groups, i.e.\ $C$ canonically lifts to an enriched comonad on the enriched category $\Modu^\ZZ R$. We are thus in a situation to apply the following graded analogue of the Eilenberg-Watts theorem.  

\begin{prop}\label{prop:EW}
Let $F\colon \Modu^\ZZ R\to \Modu^\ZZ S$ be an enriched, i.e.\ additive and grading preserving, colimit preserving functor. Then $FR$ is a graded $R$-$S$-bimodule and we have $F \cong -\otimes_R FR$.
\end{prop}
\begin{proof}
The argument is standard so we only give a sketch. By definition $FR$ is a right $S$-module. The enrichment provides the left action of $R$ via the map
\[
R = \hom_R(R,R) \to \hom_S(FR,FR)
\]
which also yields compatibility of the left and right actions.

For $M\in \Modu^\ZZ R$ we have a map $M = \hom_R(R,M) \to \hom_S(FR, FM)$ which corresponds via adjunction to the action map
\[ 
M\otimes_R FR \to FM.
\]
This is natural in $M$ and so yields an enriched natural transformation $-\otimes_R FR \to F$. It is an isomorphism at the generators $R(i)$ of $\Modu^\ZZ R$, and everything is compatible with colimits, so it is a natural isomorphism.
\end{proof}

So, given our comonad $C$, we learn that
\[
C \cong -\otimes_R C(R)
\]
for the graded $R$-$R$-bimodule $C(R)$. Next we consider the comonad structure, i.e.\ the maps $\varepsilon\colon C\to \Id$ and $\Delta\colon C\to C^2$. By the above isomorphism these can be viewed as elements of
\[
\Hom(-\otimes_R C(R), -\otimes_R R) \text{ and } \Hom(-\otimes_R C(R), -\otimes_R C(R) \otimes_R C(R))
\]
respectively. It happens that we understand these morphism sets.

\begin{lem}\label{lem:tensorembedding}
Let $X$ and $Y$ be graded $R$-$R$-bimodules. Then the natural map
\[
\Hom(X,Y) \stackrel{\gamma}{\to} \Hom(-\otimes_R X, -\otimes_R Y) 
\]
given by $f\mapsto -\otimes_R f$ is an isomorphism.
\end{lem}
\begin{proof}
It's clear that $\gamma$ has trivial kernel: if $\gamma(f)=0$ then $\gamma(f)_R\colon R\otimes_R X\to R\otimes_R Y$ is trivial, and this is just $f$ up to natural isomorphisms.

In order to show surjectivity it is enough to prove that given $\phi\in \Hom(-\otimes_R X, -\otimes_R Y)$ we have $\phi = \gamma(\phi_R)$. Let $M$ be a graded $R$-module and choose a surjection
\[
\bigoplus_{i\in \ZZ} R(i)^{\alpha_i} \to M.
\]
Tensoring with $X$ and $Y$ we get, using colimit and grading shift preservation, a commutative diagram
\[
\xymatrix{
\oplus_i X(i)^{\alpha_i} \ar[r] \ar[d]_-{\oplus_i \phi_R(i)^{\alpha_i}} & M\otimes_R X \ar[d]<1ex>^-{\phi_M} \ar[d]<-1ex>_-{M\otimes_R \phi_R}  \\
\oplus_i Y(i)^{\alpha_i} \ar[r] & M\otimes_R Y 
}
%\begin{tikzcd}
%\oplus_i X(i)^{\alpha_i} \ar{r} \ar{d}[swap]{\oplus_i \phi_R(i)^{\alpha_i}} & M\otimes_R X \arrow{d}[shift left=1.75ex]{\phi_M}  \\
%\oplus_i Y(i)^{\alpha_i} \ar{r} & M\otimes_R Y \arrow[u,shift right=1.75ex, "M\otimes_R \phi_R"]
%\end{tikzcd}
\]
where the horizontal maps are still surjective. It follows that $\phi_M = M\otimes_R \phi_R$ as required.
\end{proof}

From the lemma we learn that the comonad structure on $C \cong -\otimes_R C(R)$ is determined by morphisms of bimodules $C(R) \to R$ and $C(R) \to C(R)\otimes_R C(R)$, i.e.\ there is a comonoid structure on $C(R)$ in the monoidal category of $R$-$R$-bimodules, and we recover $C$, as a comonad, via $-\otimes_R C(R)$.

\begin{rem}
This tells us that $C(R)$ is a coring. In many cases of interest $R$ will be a graded field. Provided $R$ acts symmetrically on $C(R)$ we then have that $C(R)$ is a graded $R$-coalgebra.
\end{rem}

We can summarise our discussion as follows.

\begin{thm}\label{thm:thepoint}
Let $C$ be a colimit and grading shift preserving comonad on $\Modu^\ZZ R$. Then $C(R)$ is a graded coring, and $C$ is naturally isomorphic to $-\otimes_R C(R)$ as a comonad. The category $\Comodu C$ of comodules over the comonad $C$ is equivalent to the category of $R$-linear graded comodules $\Comodu^\ZZ_R C(R)$ over the coring $C(R)$.
\end{thm}
\begin{proof}
We have already proved the first statement. Given this we are reduced to a straightforward exercise in translation. %If one is skeptical, a concrete approach is given by noticing that the cofree objects are the same.
\end{proof}

%--------------------------------------------------------------------------------------------------------------------------------------------------

%--------------------------------------------------------------------------------------------------------------------------------------------------

\section{The monoidal structure}\label{sec:monoidal}
\setcounter{subsection}{1}

The homological residue fields of \cite{BKS} are not only Grothendieck categories. They carry a compatible symmetric monoidal structure reflecting the original tensor structure on the triangulated category. This additional structure is not explained by the equivalence of \cite{BC20} with comodules over the comonad corresponding to a tt-residue field. In order to clarify this the correct context in which to work is the following.

\begin{defn}\label{defn:mc}
Let $(C, \Delta, \varepsilon)$ be a comonad on a monoidal category $\sfM$. Then $C$ is a \emph{monoidal comonad} if $C$ is a lax monoidal functor and $\varepsilon$ and $\Delta$ are monoidal natural transformations.
\end{defn}

Let us fix a monoidal comonad as in the definition. The Eilenberg-Moore category of comodules $\Comodu C$ inherits a monoidal structure from $\sfM$: given $(M_i,\rho_i)$ for $i\in \{1,2\}$ we define a comodule structure on $M_1\otimes M_2$ via
\[
\xymatrix{
M_1 \otimes M_2 \ar[rr]^-{\rho_1\otimes \rho_2} && CM_1 \otimes CM_2 \ar[r] & C(M_1\otimes M_2)
}
\]
where the second map is courtesy of the lax monoidal structure. The various coherence conditions follow from those for the lax monoidal structure of $C$ and the fact that the comultiplication and counit are monoidal natural transformations.

\begin{ex}
Suppose that we are in the situation of \cite{BC20}, i.e.\ $\mcT$ is a big tt-category and $F\colon \mcT \to \mcK$ is an exact coproduct preserving monoidal functor to a tt-field. Then $F$ admits a coproduct preserving right adjoint $U$ which is automatically lax monoidal. Thus $FU$ is lax monoidal, and hence so is the induced comonad on $\Modu \mcK^c$.

This observation explains the monoidal structure on the homological residue field in terms of the comonadic description of \cite{BC20}*{Theorem~4.2}.
\end{ex}

From the algebraic point of view of Theorem~\ref{thm:thepoint} we expect a monoidal structure on a category of comodules over a coring to correspond to a bialgebra structure. This is exactly what we get.

\begin{cor}\label{cor:monoidal}
Let $R$ be a graded commutative ring and let $C$ be a colimit and grading shift preserving comonad on $\Modu^\ZZ R$. Suppose in addition that $C$ is a monoidal comonad. Then $C(R)$ is a graded bialgebra and the equivalence of Theorem \ref{thm:thepoint} is monoidal.
\end{cor}
\begin{proof}
The lax monoidal structure is given (up to rearranging terms using the symmetry constraint) by natural transformations
\[
\Hom(-\otimes_R R, -\otimes_R C(R)) \text{ and } \Hom(-\otimes_R C(R)\otimes_R C(R), -\otimes_R C(R)).
\]
By Lemma~\ref{lem:tensorembedding} these correspond to bimodule morphisms
\[
\eta\colon R \to C(R) \text{ and } \nabla\colon C(R)\otimes_R C(R) \to C(R),
\]
which give an algebra structure by virtue of the coherence conditions for a lax monoidal functor. They are coring maps by virtue of corresponding to monoidal natural transformations.

\end{proof}

%--------------------------------------------------------------------------------------------------------------------------------------------------

%--------------------------------------------------------------------------------------------------------------------------------------------------

\section{The examples}

We now consider three concrete cases of interest, namely the residue fields arising in algebraic geometry, stable homotopy theory, and modular representation theory.

%--------------------------------------------------------------------------------------------------------------------------------------------------

\subsection{Algebraic geometry}\label{ssec:alg}

Our starting point is a noetherian local ring $(R,\mfm,k)$. This is sufficient generality: given a point on a locally noetherian scheme we can always pass to the stalk to arrive in this situation and the residue field factors through this localization. In fact we could even complete if we wished to.

Consider the adjunction induced by $\phi\colon R\to k$, namely
\[
\phi^*\colon \sfD(R) \xymatrix{\ar[r]<0.5ex> \ar@{<-}[r]<-0.5ex> & } \sfD(k)\colon \phi_*
\]
and note that we omit the decoration indicating $\phi^*$ is derived. It is shown in \cite{BC20} that the homological residue field is given by comodules over the comonad $\phi^*\phi_*$ on $\sfD(k)$. We exploit the fact that $\sfD(k)$ is abelian: it is equivalent to $\Modu^\ZZ k$, and under this equivalence the grading shift corresponds to the suspension. Both $\phi^*$ and $\phi_*$ preserve colimits, and are compatible with the grading by virtue of commuting with suspension. The same is then true of the composite, i.e.\ $\phi^*\phi_*$ is colimit preserving and compatible with the grading.

We are thus in position to invoke Theorem~\ref{thm:thepoint} and Corollary~\ref{cor:monoidal}. The bialgebra in question is the graded $k$-bialgebra
\[
\phi^*\phi_* k = \Tor^R_*(k,k)
\] 
and so the homological residue field of $R$ at $k$ is the category of graded comodules over $\Tor^R_*(k,k)$, together with the tensor product induced by the algebra structure. We see that the homological residue field packages the tangent information recorded by $\phi^*$ into the structure of the category.

\begin{rem}\label{rem:algHopf}
In fact $\Tor^R_*(k,k)$ carries more information. It is a graded Hopf algebra. This corresponds to the fact that the finite dimensional modules are dualizable (i.e.\ that $\sfD(k)$ is abelian).
\end{rem}

%--------------------------------------------------------------------------------------------------------------------------------------------------

\subsection{Topology}\label{ssec:top}

There is an analogous story for the residue fields of the stable homotopy category $\SH$. Fixing a prime $p$ and an $n\geq 0$ we can consider the Morava $K$-theory spectrum $K(p,n)$ and the associated adjunction
\[
F\colon \SH \xymatrix{\ar[r]<0.5ex> \ar@{<-}[r]<-0.5ex> & } \Ho(\Modu K(p,n))\colon U
\]
between the stable homotopy category and the homotopy category of $K(p,n)$-module spectra. See \cite{Ravenel92} for a discussion of Morava $K$-theory and its role as a residue field for $\SH$. Since the coefficients of Morava $K$-theory $K(p,n)_*$ form a graded field the category $\Ho(\Modu K(p,n))$ is equivalent to the semisimple abelian category $\Modu^\ZZ K(p,n)_*$. 

By the same reasoning as in the previous section (cf.\ \cite{BC20}*{Corollary~4.4}) we can apply Theorem~\ref{thm:thepoint} and Corollary~\ref{cor:monoidal}. We see that the homological residue field in this case is the category of graded comodules over the graded bialgebra $K(p,n)_*(K(p,n))$ of stable cooperations on the relevant Morava $K$-theory.

\begin{rem}\label{rem:topHopf}
As in Remark~\ref{rem:algHopf} the bialgebra $K(p,n)_*(K(p,n))$ is actually a Hopf algebra.
\end{rem}

%--------------------------------------------------------------------------------------------------------------------------------------------------

\subsection{Modular representation theory}\label{ssec:rep}

In the case of modular representation theory the situation is richer, and so we are more modest in our aims. In fact this setting differs in a few key points, not least of which is the fact that our tt-residue fields are, in general, no longer semisimple abelian categories.

Fix a prime $p$ and consider an elementary abelian $p$-group $E$, of rank $r$, and a field $k$ of characteristic $p$. Let $\cyc$ denote the cyclic group with $p$ elements. We will usually think of $kE$ as the truncated polynomial ring $kE \cong k[x_1,\ldots,x_r]/(x_1^p,\ldots,x_r^p)$, and similarly $k\cyc \cong k[t]/(t^p)$.

Pick a $k$-point $a\in \PP^{r-1}_k$ and consider the map defined by the corresponding cyclic shifted subgroup
\[
\phi \colon k\cyc \to kE \text{ via } t \mapsto \sum_{i=1}^r a_ix_i.
\]
Viewing $k\cyc$ and $kE$ as Hopf algebras via their incarnation as enveloping algebras of $p$-restricted Lie algebras, i.e.\
\[
\Delta(t) = t\otimes 1 + 1\otimes t \text{ and } S(t) = -t
\]
and similarly for $kE$, the map $\phi$ is a map of Hopf algebras. Thus $\phi_*\colon \sModu kE \to \sModu k\cyc$ is a monoidal exact functor giving rise to a tt-residue field for the point 
\[
a\in \PP^{r-1}_k \cong \Spc \smodu kE.
\]

If we allow field extensions of $k$ then all homological residue fields arise from such cyclic shifted subgroups, i.e.\ to parametrize the homological (and tensor triangulated) spectrum we use the entire scheme $\PP^{r-1}_k$ not just the $k$-points. For simplicity of presentation we ignore field extensions here.

If $p\neq 2$ then $\smodu k\cyc$ is not abelian, but it is still quite manageable. The description involves the preprojective algebra of a quiver; we refer to \cite{MR1648647} for the definition and several equivalent descriptions, and we give a presentation of the algebras relevant to us in Remark~\ref{rem:ppa}.

\begin{lem}
The stable module category $\smodu k\cyc$ is Morita equivalent to $\Pi = \Pi_{A_{p-1}}$ the preprojective algebra of Dynkin type $A_{p-1}$, i.e.\
\[
\Modu \smodu k\cyc \cong \Modu \Pi.
\]
The suspension on $\smodu k\cyc$ corresponds to the automorphism of $\Pi$ induced by reflection about the central arrow of $A_{p-1}$ (with the convention that this is the identity when $p=2$).
\end{lem}
\begin{proof}
Identifying $k\cyc$ with $k[t]/(t^p)$ as above, the stable category can be described as 
\[
\smodu k\cyc = \add(k[t]/(t^i) \mid 1\leq i \leq p-1).
\]
Thus the statement boils down to computing the endomorphism ring of $\oplus_i k[t]/(t^i)$ in the stable category, which is well known to be $\Pi$. 
\end{proof}

\begin{rem}\label{rem:ppa}
The algebra $\Pi$ can be described by the quiver
\[
\xymatrix{
1 \ar[r]<0.5ex>^-{\alpha_1} \ar@{<-}[r]<-0.5ex>_-{\beta_1} & 2 \ar[r]<0.5ex>^-{\alpha_2} \ar@{<-}[r]<-0.5ex>_-{\beta_2} & \cdots \ar[r]<0.5ex>^-{\alpha_{p-3}} \ar@{<-}[r]<-0.5ex>_-{\beta_{p-3}} & p-2 \ar[r]<0.5ex>^-{\alpha_{p-2}} \ar@{<-}[r]<-0.5ex>_-{\beta_{p-2}} & p-1
}
\]
with the relations
\[
\beta_1\alpha_1 = 0 = \alpha_{p-2}\beta_{p-2} \text{ and } \alpha_i\beta_i = \beta_{i+1}\alpha_{i+1} \text{ for } 2\leq i \leq p-3.
\]
In this context $\Pi$ also goes by the moniker \emph{stable Auslander algebra} of $kA$. 
\end{rem}

So, we are led to consider the comonad $C$ on $\Modu \smodu k\cyc \cong \Modu \Pi$ induced by the adjunction
\[
\phi_*\colon \sModu kE \xymatrix{\ar[r]<0.5ex> \ar@{<-}[r]<-0.5ex> & } \sModu kA\colon \phi^!
\]
where $\phi^!$ denotes the coinduction functor. In other words, $C$ is base change along $\phi_*\phi^!\colon \smodu kA \to \smodu kA$, and so preserves colimits. Since $\phi_*$ and $\phi^!$ are triangulated functors they commute with suspension, and so $C$ commutes with the involution on $\Modu \Pi$ induced by reflection about the central arrow. Thus we can apply the (usual!) Eilenberg-Watts theorem, and the ungraded analogue of Lemma~\ref{lem:tensorembedding} to see that $C$ is determined by the coring $C(\Pi)$, i.e.\ by a coring structure on the image of
\[
\bigoplus_{i=1}^{p-1}\phi_*\phi^!k[t]/(t^i) = \bigoplus_{i=1}^{p-1} \phi_*\Hom_{k\cyc}(kE, k[t]/(t^i)) = \bigoplus_{i=1}^{p-1} k[t]/(t^i)^{\oplus {p^{r-1}}}
\]
under the Yoneda embedding. 

%--------------------------------------------------------------------------------------------------------------------------------------------------

\subsubsection{The case $p=2$}

In this case $\Pi = \Pi_{A_1} = k$ and we are working in the category of $k$-vector spaces. The coalgebra of interest is, as a vector space, 
\[
C(\Pi) = k^{\oplus 2^{r-1}}.
\]
It is natural to guess it is the (ungraded!) exterior bialgebra on $r-1$ generators, and we claim this is the case. In fact, this is easily checked directly by making a change of coordinates and working with the adjunctions corresponding to the map of $k$-algebras
\[
k[x]/(x^2) \to k[x,y_1,\ldots,y_{r-1}]/(x^2,y_1^2,\ldots,y_{r-1}^2).
\]

Thus the homological residue field at a $k$-point of $\Spc \smodu kE$ is equivalent to the category of comodules over the exterior $k$-coalgebra in $r-1$ variables, where $r$ is the rank of $E$.

\begin{rem}
Again this is actually a Hopf algebra, which corresponds to the fact that when $p=2$ the tt-residue field is abelian.
\end{rem}

\begin{rem}
This is expected, via Koszul duality and taking orbit categories, from Section~\ref{ssec:alg} applied to a rational point on $\PP_k^{r-1}$. Let us elaborate a little: in characteristic $2$ the group algebra $kE$ of an elementary abelian $2$-group $E$ of rank $r$ is an exterior algebra on $r$ generators. This is Koszul with its standard grading and there is an associated equivalence of categories $\smodu^\ZZ kE \cong \sfD^\mathrm{b}(\PP^{r-1}_k)$ where $\smodu^\ZZ kE$ denotes the stable category of graded $kE$-modules (see \cite{MR4103346}*{Theorem~2.3.3} for a general statement and further references). Under this equivalence the grading shift on $\smodu^\ZZ kE$ corresponds to tensoring with $\Sigma\mcO(-1)$ and so $\smodu kE$ is equivalent to the orbit category of $\sfD^\mathrm{b}(\PP^{r-1}_k)$ by the latter functor. 

If we focus our attention at a single rational point on $\PP^{r-1}_k$ then the homological residue field corresponds to the graded exterior $k$-coalgebra in $r-1$ variables. Passing to the orbit category, in the neighbourhood of a point, simply makes the suspension the identity. This has the effect of removing the grading on the homological residue field.
\end{rem}

\subsection{Rational and even periodic $E_{\infty}$-rings}
Mathew has constructed tt-residue fields for the homotopy category of modules over a rational noetherian $E_{\infty}$-ring spectrum containing a unit in degree $2$ \cite{mathewresiduefields}, and over a (not necessarily rational) even periodic $E_\infty$-ring  spectrum with $\pi_0$ regular Noetherian \cite{mathewthicksubcategory}. Let $R$ be a ring spectrum fitting either of these descriptions.

  Mathew proves there is a residue field $\sfD(\kappa(\mathfrak{p}))$ for each prime $\mathfrak{p}$ of $\pi_0(R)$, where $\kappa(\mathfrak{p})$ is an $R$-algebra whose homotopy is a graded field, and the residue field functor takes an $R$-module $M$ to $\pi_*(M \otimes_R \kappa(\mathfrak{p}))$. So our framework applies: if $R$ is a such a ring spectrum with residue field $\kappa(\mathfrak{p})$, then the bialgebra giving the homological residue field is again $\pi_*(\kappa(\mathfrak{p}) \otimes_R \kappa(\mathfrak{p}))$, i.e.\ $\Tor^R_*(\kappa(\mathfrak{p}),\kappa(\mathfrak{p}))$.

In both cases Mathew shows that these residue fields collectively detect nilpotence. We conclude from \cite{Balmer20b}*{Theorem~5.4} that this describes all homological residue fields of $\sfD(R)$, under the assumptions on the homotopy of $R$. We do not know of a broader class of $E_{\infty}$ rings for which a classification of residue fields or homological residue fields is known-- often the construction of appropriate residue fields is a difficult problem.

\subsection{The closed point of cochains on $BG$} We conclude this section with another special case of interest. Let $G$ be a finite group, $BG$ the classifying space, and $k$ a field of characteristic dividing the order of $G$. Then we can consider the ring spectrum $R = C^*(BG;k)$ which comes with an augmentation, corresponding to a choice of base point, $\phi\colon R \to k$. This, analogously to the situation of a local ring, gives rise to the tt-residue field
\[
\phi^*\colon \sfD(R) \xymatrix{\ar[r]<0.5ex> \ar@{<-}[r]<-0.5ex> & } \sfD(k)\colon \phi_*
\]
with associated bialgebra
\[
\Tor_*^R(k,k) \cong \mathrm{H}_*(\Omega(BG^{\wedge}_p))^*
\]
the graded dual of the homology of loops on the p-completion of $BG$.

%--------------------------------------------------------------------------------------------------------------------------------------------------

%--------------------------------------------------------------------------------------------------------------------------------------------------

%--------------------------------------------------------------------------------------------------------------------------------------------------

\section{Change of tt-residue field: an example}

Suppose that $\mcT$ is a big tt-category and we are given an exact monoidal functor to a tt-field $\mcF$, say $F\colon \mcT\to \mcF$, with right adjoint $U$. We could then consider another exact monoidal functor $G\colon \mcF \to \mcL$ with $\mcL$ also a tt-field (a tt-field extension if you will) and right adjoint $V$. The theory, namely \cite{BC20}*{Theorem~4.2}, tells us that the categories of comodules over the comonads corresponding to $FU$ and to $GFUV$, on $\Modu \mcF^c$ and $\Modu \mcL^c$ respectively, must be equivalent: they are both describing the corresponding homological residue field.

In this short section we spell out, in a very naive way, how this is realized and give a concrete example in the context of algebraic geometry. All of this is, more or less, standard descent theory.

%--------------------------------------------------------------------------------------------------------------------------------------------------

\subsection{Generalities}

Suppose that we are given three adjunctions
\[
\sfA \xymatrix{\ar[r]<0.5ex>^-F \ar@{<-}[r]<-0.5ex>_-U & } \sfB \xymatrix{\ar[r]<0.5ex>^-G \ar@{<-}[r]<-0.5ex>_-V & } \sfC
\]
i.e.\ $(F,U)$, $(G,V)$ and $(GF,UV)$, all of which are comonadic. Thus we have equivalences
\[
\Comodu GFUV \cong \sfA \cong \Comodu FU \text{ and } \sfB \cong \Comodu GV
\]
giving a description of $\sfB$ and two descriptions of $\sfA$. From the equivalence $\sfB \cong \Comodu GV$ we deduce a comonad $D$ on $\Comodu GV$ corresponding to $FU$ on $\sfB$. Then we have
\[
\sfA \cong \Comodu D
\]
describing $\sfA$ as an iterated comodule category; we note that we are taking $D$-comodules in the category of $GV$-comodules, and so we can think of the coaction of $D$ as being $GV$-linear.

%--------------------------------------------------------------------------------------------------------------------------------------------------

\subsection{A concrete example}

Let us consider the noetherian local ring $\RR[[x]]$ and the pair of adjunctions
\[
\sfD(\RR[[x]]) \xymatrix{\ar[r]<0.5ex>^-{\phi^*} \ar@{<-}[r]<-0.5ex>_-{\phi_*} & } \sfD(\RR) \xymatrix{\ar[r]<0.5ex>^-{i^*} \ar@{<-}[r]<-0.5ex>_-{i_*} & } \sfD(\CC)
\]
giving rise to two suitable functors to tt-fields capturing the information at the closed point $(x)$ of $\RR[[x]]$. 

We know from Section~\ref{ssec:alg} that using the residue field $\RR$ describes the homological residue field as the category of $\ZZ$-graded comodules $\Comodu^\ZZ \RR[\tau]/(\tau^2)$ where
\[
\Tor^{\RR[[x]]}_*(\RR,\RR) \cong \RR[\tau]/(\tau^2) \text{ with } \vert \tau \vert = -1 \text{ and } \Delta(\tau) = \tau \otimes 1 + 1\otimes \tau
\]
is the graded exterior coalgebra.

It is then natural to ask how to translate the abstract description of this category in terms of $\sfD(\RR[[x]]) \to \sfD(\CC)$, as discussed above, into something similarly concrete.

Let $G$ denote the Galois group of the extension $\RR\to \CC$. Then classical Galois descent tells us that $\Modu \RR$ is the category of comodules in $\Modu \CC$ for the Galois coring $C = \CC \otimes_\RR \CC$ with counit
%\[
%\CC G \stackrel{\varepsilon}{\to} \CC \text{ given by } \varepsilon(1) = 1, \varepsilon(\sigma) = i
%\] 
the multiplication map $\CC\otimes_\RR\CC \to \CC$ and comultiplication
\[
\CC\otimes_\RR\CC \to \CC\otimes_\RR\CC\otimes_\CC\CC\otimes_\RR\CC \text{ by } \Delta(\alpha\otimes_\RR \beta) = (\alpha\otimes_\RR 1)\otimes_\CC(1\otimes_\RR \beta).
\]
(The interested reader can consult \cite{MR2075585} for a survey on descent from this point of view, and also for further references.) This carries over to the graded context (with trivial grading on $C$), and so 
\[
\Modu^\ZZ\RR \cong \Comodu^\ZZ C.
\]

In this context, passing to the iterated comodule description just boils down to tensoring our corings: we consider the coring over $\CC$ given by 
\[
\Tor^{\RR[[x]]}_*(\RR,\RR) \otimes_\RR \CC \otimes_\RR \CC = \RR[\tau]/(\tau^2) \otimes_\RR C.
\]
Then we get two descriptions of the homological residue field of $\sfD(\RR[[x]])$ at the closed point
\[
\Comodu^\ZZ_\RR \RR[\tau]/(\tau^2) \cong \Comodu^\ZZ_\CC \RR[\tau]/(\tau^2) \otimes_\RR C
\]
where the subscripts indicate the ambient category. Essentially, the right-hand side adds the Galois-group equivariance necessary to cut back down from $\CC$ to $\RR$.

\begin{rem}
One could, as is coming up in Section~\ref{sec:duality}, take the dual picture. Then in terms of $\RR$ the homological residue field would be graded modules for the standard graded exterior algebra $\Ext^*_{\RR[[x]]}(\RR,\RR)$, and in terms of $\CC$ it would be graded modules for the group algebra
\[
(\Ext^*_{\RR[[x]]}(\RR,\RR)\otimes_\RR \CC)G
\]
i.e.\ graded modules over the complex exterior algebra together with an action of the Galois group $G$ providing descent data.
\end{rem}

%--------------------------------------------------------------------------------------------------------------------------------------------------

%--------------------------------------------------------------------------------------------------------------------------------------------------

\section{From coalgebras to algebras}\label{sec:duality}

\setcounter{subsection}{1}

Given sufficient finiteness conditions we can move freely between the theory of coalgebras and that of algebras. Nothing here is new or non-standard, but in some cases it leads us to a more familiar vantage point.

In this section we work over a field $k$. Let $C$ be a colimit preserving $k$-linear comonad on $\Modu^\ZZ k$ (i.e.\ $C$ preserves arbitrary coproducts) compatible with the grading. By Theorem~\ref{thm:thepoint} we have a  corresponding graded coalgebra $C(k)$. We denote by $\comodu^\ZZ C(k)$ the category of finite dimensional graded $C(k)$-comodules, and by $\Comodu^\ZZ C(k)$ the category of all graded $C(k)$-comodules.
%In case I get worried k-linearity gives us that the action of k is central so it is really a k-algebra

We state a couple of facts that will be used in the sequel.

\begin{lem}\label{lem:dual}
The graded $k$-dual $A = C(k)^*$ of $C(k)$ is naturally a graded $k$-algebra. 
\end{lem}

\begin{lem}\label{lem:ind}
The ind-completion of $\comodu^\ZZ C(k)$ is $\Comodu^\ZZ C(k)$.
\end{lem}
\begin{proof}
This follows from \cite{Waterhouse}*{3.3}. 
\end{proof}

Suppose that $C(k)$ is finite dimensional, i.e.\ $C$ restricts to $\modu^\ZZ k$. Starting from a finite dimensional graded right comodule $V$ we can dualize to obtain a graded left $A$-module (where $A = C(k)^*$). Taking the $k$-dual again we produce a graded right $A$-module, and this gives a covariant functor
\[
\Phi\colon \comodu^\ZZ C(k) \to \modu^\ZZ A.
\]
Since $A$ is finite dimensional this process is reversible: from  a finite dimensional right $A$-module $M$ we can dualize to a left module, and then dualize again to a right $A^* = C(k)^{**} \cong C(k)$-comodule. Thus $\Phi$ is an equivalence of categories. 

\begin{prop}\label{prop:equiv}
Let $C(k)$ be finite dimensional, with dual algebra $A$. The functor $\Phi$ induces an equivalence
\[
\Comodu^\ZZ C(k) \to \Modu^\ZZ A.
\]
\end{prop}
\begin{proof}
We know $\Phi$ is an equivalence $\comodu^\ZZ C(k) \to \modu^\ZZ A$ and so it induces an equivalence on ind-completions. The ind-completion of $\comodu^\ZZ C(k)$ is $\Comodu^\ZZ C(k)$. As $A$ is finite dimensional, the category $\modu^\ZZ A$ consists precisely of the finitely presented graded $A$-modules. Thus the ind-completion of $\modu^\ZZ A$ is $\Modu^\ZZ A$ and we are done.
\end{proof}

%\begin{rem}
%\Greg{There might be some game in the locally finite case. One still has duality, but tensor products can be a problem as they won't preserve local finiteness in general. It is ok for connective or coconnective guys though, so one can probably always use the duality between Tor and Ext in the case of local rings to get something.}
%\end{rem}

\begin{rem}
If one is willing to take into account a topology this always works. One can describe $\Comodu^\ZZ C(k)$ as the category of graded modules over a pseudocompact graded algebra. However, in our examples this does not seem to lead to a more psychologically comforting description of the residue fields.
\end{rem}

%--------------------------------------------------------------------------------------------------------------------------------------------------

\subsection{Example}

Let $(R,\mfm,k)$ be a regular local ring and let $\phi$ denote the quotient map $R\to k$. We let $\Lambda$ denote $\Ext^*_R(k,k)$ which, as $R$ is regular, is the exterior algebra on the tangent space viewed as a graded vector space in degree $1$.

\begin{thm}
The homological residue field of $\sfD(R)$ at $\mfm$ is $\Modu^\ZZ \Lambda$. 
\end{thm}
\begin{proof}
By Theorem~\ref{thm:thepoint} the homological residue field is equivalent to $\Comodu^\ZZ \Tor^R_*(k,k)$. As $R$ is regular the coalgebra $\Tor^R_*(k,k)$ is finite dimensional. The dual algebra is $\Lambda$, and Proposition~\ref{prop:equiv} tells us that $\Comodu^\ZZ \Tor^R_*(k,k) \cong \Modu^\ZZ \Lambda$.
\end{proof}

%--------------------------------------------------------------------------------------------------------------------------------------------------

%--------------------------------------------------------------------------------------------------------------------------------------------------

%--------------------------------------------------------------------------------------------------------------------------------------------------

%--------------------------------------------------------------------------------------------------------------------------------------------------

\bibliography{greg_bib}

\def\cprime{$'$}
% \bib, bibdiv, biblist are defined by the amsrefs package.
\begin{bibdiv}
\begin{biblist}

\bib{balmerhomologicalsupports}{article}{
      author={Balmer, Paul},
       title={Homological support of big objects in tensor-triangulated
  categories},
        date={2020},
     journal={Journal de l'{\'E}cole polytechnique---Math{\'e}matiques},
      volume={7},
       pages={1069\ndash 1088},
}

\bib{Balmer20b}{article}{
      author={Balmer, Paul},
       title={Nilpotence theorems via homological residue fields},
        date={2020},
     journal={Tunisian Journal of Mathematics},
      volume={2},
      number={2},
       pages={359\ndash 378},
}

\bib{BC20}{article}{
      author={Balmer, Paul},
      author={Cameron, James~C},
       title={Computing homological residue fields in algebra and topology},
        date={2021},
     journal={Proc. Amer. Math. Soc.},
      volume={149},
      number={8},
       pages={3177\ndash 3185},
}

\bib{BKS}{article}{
      author={Balmer, Paul},
      author={Krause, Henning},
      author={Stevenson, Greg},
       title={Tensor-triangular fields: ruminations},
        date={2019},
     journal={Selecta Math. (N.S.)},
      volume={25},
      number={1},
       pages={Paper No. 13, 36},
}

\bib{MR2075585}{incollection}{
      author={Caenepeel, S.},
       title={Galois corings from the descent theory point of view},
        date={2004},
   booktitle={Galois theory, {H}opf algebras, and semiabelian categories},
      series={Fields Inst. Commun.},
      volume={43},
   publisher={Amer. Math. Soc., Providence, RI},
       pages={163\ndash 186},
}

\bib{MR4103346}{article}{
      author={Gratz, Sira},
      author={Stevenson, Greg},
       title={Homotopy invariants of singularity categories},
        date={2020},
     journal={J. Pure Appl. Algebra},
      volume={224},
      number={12},
       pages={106433, 18},
}

\bib{mathewthicksubcategory}{article}{
      author={Mathew, Akhil},
       title={A thick subcategory theorem for modules over certain ring
  spectra},
        date={2015},
        ISSN={1465-3060},
     journal={Geom. Topol.},
      volume={19},
      number={4},
       pages={2359\ndash 2392},
         url={https://mathscinet.ams.org/mathscinet-getitem?mr=3375530},
}

\bib{mathewresiduefields}{article}{
      author={Mathew, Akhil},
       title={Residue fields for a class of rational
  {$\mathbf{E}_\infty$}-rings and applications},
        date={2017},
        ISSN={0022-4049},
     journal={J. Pure Appl. Algebra},
      volume={221},
      number={3},
       pages={707\ndash 748},
         url={https://mathscinet.ams.org/mathscinet-getitem?mr=3556705},
}

\bib{Ravenel92}{book}{
      author={Ravenel, Douglas~C.},
       title={Nilpotence and periodicity in stable homotopy theory},
      series={Annals of Mathematics Studies},
   publisher={Princeton University Press, Princeton, NJ},
        date={1992},
      volume={128},
        ISBN={0-691-02572-X},
        note={Appendix C by Jeff Smith},
      review={\MR{1192553}},
}

\bib{MR1648647}{incollection}{
      author={Ringel, Claus~Michael},
       title={The preprojective algebra of a quiver},
        date={1998},
   booktitle={Algebras and modules, {II} ({G}eiranger, 1996)},
      series={CMS Conf. Proc.},
      volume={24},
   publisher={Amer. Math. Soc., Providence, RI},
       pages={467\ndash 480},
}

\bib{Waterhouse}{book}{
      author={Waterhouse, William~C.},
       title={Introduction to affine group schemes},
      series={Graduate Texts in Mathematics},
   publisher={Springer-Verlag, New York-Berlin},
        date={1979},
      volume={66},
}

\end{biblist}
\end{bibdiv}

\end{document}